\documentclass[runningheads]{llncs}

\usepackage{url, hyperref}

\usepackage[T1]{fontenc}
\usepackage{amssymb}
\usepackage{bm}
\usepackage{graphicx}
\usepackage{psfrag}
\usepackage{color}
\usepackage{amsmath,color}
\usepackage{enumerate}
\usepackage{algpseudocode} 
\usepackage{lmodern}
\usepackage{todonotes}
\usepackage{xy}
\input xy
\xyoption{all}

\newcommand{\excise}[1]{}

\newcommand{\ve}{\boldsymbol}

\newtheorem{thm}{Theorem}
\newtheorem{cor}[thm]{Corollary}

\newcommand{\be}{\begin{eqnarray}}
\newcommand{\bea}{\begin{eqnarray*}}
\newcommand{\ee}{\end{eqnarray}}
\newcommand{\eea}{\end{eqnarray*}}

\newcommand{\diag}{\mathrm{diag}}

\newcommand{\lin}{\mathrm{lin}}

\newcommand{\conv}{\mathrm{conv}}
\newcommand{\vol}{\mathrm{vol}}

\newcommand{\R}{\mathbb R}
\newcommand{\Z}{\mathbb Z}

\newcommand{\KP}{{P}}

\newcommand\supp{\mathrm{supp}}

\newcommand\bb{\ve b}%
\newcommand\cc{\ve c}%
\newcommand\xx{\ve x}%
\newcommand\zz{\ve z}%

\newcommand\BB{\ve B}%
\renewcommand\AA{\ve A}%
\renewcommand\SS{\ve S}

\newcommand\DD{\ve D}%
\newcommand\OO{\ve O}%

\newcommand\IG{\mathrm{IG}}%
\newcommand\IP{\mathrm{IP}}%
\newcommand\LP{\mathrm{LP}}%
\newcommand{\norm}[1]{\left\Vert #1\right\Vert }%
\newcommand\AAhat{\hat{\AA\;}\negthickspace}%

\def\udots{\mathinner{\mkern1mu\raise1pt\vbox{\kern7pt\hbox{.}}\mkern2mu
\raise4pt\hbox{.}\mkern2mu\raise7pt\hbox{.}\mkern1mu}}

\title{Sparsity and integrality gap transference bounds for integer programs}

\author{Iskander Aliev\inst{1}, Marcel Celaya\inst{1}, and Martin Henk\inst{2}}

\authorrunning{I. Aliev et al.}

\institute{Cardiff University\\
\email{{\{alievi,celayam\}@cardiff.ac.uk}}
\and Technische Universit\"at Berlin\\
\email{henk@math.tu-berlin.de}
}

\date{\today }

\begin{document}
\maketitle
\thispagestyle{plain}

\begin{abstract}
	\noindent
	We obtain new transference bounds that connect two active areas of research: proximity and sparsity of solutions to integer programs. Specifically, we study the additive integrality gap of the integer linear programs $\min\{{\ve c}\cdot{\ve x}: {\ve x}\in P\cap\Z^n\}$, where $P=\{{\ve x}\in\R^n: \AA{\ve x}={\ve b}, {\ve x}\ge {\ve 0}\}$ is a polyhedron in the standard form determined by an integer $m\times n$ matrix $\AA$ and an integer vector ${\ve b}$. The main result of the paper gives an upper bound for the integrality gap that drops exponentially in the size of support of the optimal solutions corresponding to the vertices of the integer hull of $P$. Additionally, we obtain a new proximity bound that estimates the $\ell_2$-distance from a vertex of $P$ to its nearest integer point in the polyhedron $P$. The proofs make use of the results from the geometry of numbers and convex geometry.

\end{abstract}

\section{Introduction and main results}

The proximity and sparsity of solutions to integer programs are two active areas of research in the theory of mathematical programming.


Proximity-type results study the quality of approximation of the solutions to integer programs by the solutions of linear programming relaxations. This is a traditional research topic with the first contributions dated back at least to Gomory \cite{Gomory_sensitivity,Gomory_polyhedra}. The most influential results in this area also include the proximity bounds by Cook et al. \cite{MR839604} and by Eisenbrand and Weismantel \cite{MR3775840}. For more recent contributions, we refer the reader to Celaya et al. \cite{CKPW}, Lee et al. \cite{LPSX,LPSX1}, and Paat, Weismantel and Weltge \cite{PWW}.

Sparsity-type results study the size of support of solutions to integer programs. This area of research takes its origin from the integer Carath\'eodory theorems of Cook, Fonlupt and Schrijver \cite{MR830593} and Seb\H{o} \cite{Sebo} and later major contributions by Eisenbrand and Shmonin \cite{EisenbrandShmonin2006}.  In recent years, this topic has been studied in numerous papers, including Abdi et al. \cite{ACGT}, Aliev et al. \cite{AADO,Support,ADON2017}, Berndt, Jansen and Klein \cite{BJK}, Dubey and Liu \cite{DL} and Oertel, Paat and Weismantel \cite{OPW}.  

Recent works of Lee et al. \cite{LPSX} and Aliev et al. \cite{ACHW} show that the proximity and sparsity areas are highly interconnected. The paper \cite{ACHW} establishes {\em transference bounds} that link both areas in two {\em special cases}: for  corner polyhedra and knapsacks with positive entries. Remarkably, this gives a drastic improvement on the previously known proximity bounds for knapsacks obtained in \cite{AHO}.

In this paper, we establish the first transference bounds that involve the integrality gap of integer linear programs and hold in the {\em general case}, addressing a future research question posed in  \cite{ACHW}. The proofs explore a new geometric approach that combines Minkowski's geometry of numbers and box slicing inequalities.

Specifically, 
for a matrix $\AA \in \Z^{m\times n}$ with $m<n$ and $\ve b \in \Z^m$,
%
we define the polyhedron
\bea
\KP(\AA, {\ve b})=\{{\ve x}\in \R^n_{\ge 0}: \AA{\ve x}={\ve b}\}\,
\eea
and assume that $\KP(\AA, {\ve b})$ contains integer points. 
Given a cost vector $\cc\in\R^{n}$, we consider the integer linear programming
problem 
\begin{eqnarray}
\min\{\cc\cdot\xx:\xx\in\KP(\AA,\bb)\cap\Z^{n}\}\,,\label{initial_IP}
\end{eqnarray}
where $\cc\cdot\xx$ stands for the standard inner product.
A very successful and traditional approach for solving optimisation problems 
of the form \eqref{initial_IP} is based on solving its
linear programming relaxation 
\begin{eqnarray}
\min\{\cc\cdot\xx:\xx\in\KP(\AA,\bb)\}\,,\label{initial_LP}
\end{eqnarray}
obtained
by dropping the integrality constraint. Subsequently,  various methods are used to construct a feasible integer solution to \eqref{initial_IP} from a fractional solution of \eqref{initial_LP}.

Suppose that (\ref{initial_IP}) is feasible and bounded. Let $\IP(\AA,\bb,\cc)$ and $\LP(\AA,\bb,\cc)$ denote the optimal
values of (\ref{initial_IP}) and (\ref{initial_LP}), respectively.
We will focus on estimating the  \emph{(additive) integrality gap} $\IG(\AA,\bb,\cc)$ defined as
\begin{eqnarray*}
\IG(\AA,\bb,\cc)=\IP(\AA,\bb,\cc)-\LP(\AA,\bb,\cc)\,.
\end{eqnarray*}
The integrality gap is a fundamental proximity characteristic of the problem (\ref{initial_IP}) extensively studied in the literature. The upper bounds for $\IG(\AA,\bb,\cc)$ appear already in the work of Blair and Jeroslow \cite{B_and_J_value_1,B_and_J_value}. To state the best currently known estimates, we will need to introduce the following notation. By $\|\cdot\|_1$ and $\|\cdot\|_\infty$ we denote the $\ell_1$ and $\ell_\infty$
norms, respectively. Without loss of generality, we assume that $\AA$ has rank $m$ and, for $1\le r\le m$, denote by $\Delta_r(\AA)$ the maximum absolute $r\times r$ subdeterminant of $A$, that is 
\bea
\Delta_r(\AA)=\max\{|\det \BB|: \BB \mbox{ is a }r\times r\mbox{ submatrix of }\AA\}\,.
\eea
The sensitivity theorems of Cook et al. 
\cite{MR839604} (see also Celaya et al. \cite{CKPW} for further improvements) imply the bound
\be\label{IG_C}
\IG(\AA,\bb,\cc)\le \|{\ve c}\|_1 (n-m) \Delta_m(\AA)\,.
\ee
More recently, Eisenbrand and Weismantel \cite{MR3775840} obtained the estimate
\be\label{IG_EW}
\IG(\AA,\bb,\cc)\le \|{\ve c}\|_\infty m (2m\Delta_1(\AA)+1)^m\,,
\ee
which is, remarkably, independent of the dimension $n$. 

Given a set $K\subset \R^n$ we will denote by $\conv(K)$ the convex hull of $K$. The polyhedron 
\bea
\KP_I(\AA,\bb)=\conv(\KP(\AA,\bb)\cap \Z^n)\,
\eea
is traditionally referred to as the {\em integer hull} of $\KP(\AA,\bb)$. Clearly, the problem (\ref{initial_IP}) has at least one optimal solution which is a vertex of the integer hull $\KP_I(\AA,\bb)$. 
Given such a solution $\ve z^*$, we obtain transference bounds that link the integrality gap $\IG(\AA,\bb,\cc)$ with the size of the support of  $\ve z^*$. 

Observe  that the integrality gap is positive homogeneous of degree one in ${\ve c}$, that is for $t>0$
\be\label{homogenity}
\IG(\AA,\bb,t\cc)=t\,\IG(\AA,\bb,\cc)\,.
\ee
In what follows, we also use the notation $\|\cdot\|_2$ for the $\ell_2$ norm.
In view of \eqref{homogenity}, we may assume without loss of generality that ${\ve c}$ is a unit vector, that is $\|{\ve c}\|_2=1$.

Given a vector ${\ve x}=(x_1,\ldots, x_n)^\top\in \R^n$, we will denote by $\supp({\ve x})$ the {\em support} of ${\ve x}$, that is $\supp({\ve x})=\{i: x_i\neq 0\}$. To measure the size of the support, we use the {\em $0$-norm} $\|{\ve x}\|_0=|\supp({\ve x})|$. The first result makes use of the quantity
\bea
\Delta(\AA)=\sqrt{\det \AA\AA^\top}\,.
\eea
Geometrically, $\Delta(\AA)$ is the $m$-dimensional volume of the parallelepiped determined by the rows of $\AA$.
We will also denote by $\gcd(\AA)$ the greatest common divisor of all $m\times m$ subdeterminants of $\AA$. 

\begin{thm} \label{thm_transference_gap_integer_hull} Let $\AA\in \Z^{m\times n}$ be a matrix of rank $m$, $\bb\in \Z^m$ and $\cc\in \R^n$ be a unit cost vector. Suppose that (\ref{initial_IP}) is feasible and bounded. Let $\zz^*$ be an optimal solution to (\ref{initial_IP}) which is a vertex of $\KP_I(\AA,\bb)$. Then 
 \be\label{Gap_transference_integer_hull} 
\IG(\AA,\bb,\cc)\le \frac{s}{2^{s-m-1}}\cdot\frac{\Delta(\AA)}{\gcd(\AA)}\,,
\ee
 where $s=\| {\ve z}^*\|_0$.
\end{thm}
Hence, we obtain an upper bound for the integrality gap that drops {\em exponentially} in the size of support of any optimal solution to (\ref{initial_IP}) which is a vertex of the integer hull $\KP_I(\AA,\bb)$.  We remark that  \cite{ACHW} gives optimal transference bounds for positive knapsacks and corner polyhedra that connect the $\ell_\infty$-distance proximity and the size of the support of integer feasible solutions. Theorem \ref{thm_transference_gap_integer_hull} applies in the general case and a different setting; it connects the integrality gap and the size of the support of optimal solutions to \eqref{initial_IP}.

Next, we obtain from Theorem \ref{thm_transference_gap_integer_hull}  transference bounds in terms of the maximum absolute $m\times m$ subdeterminant $\Delta_m(\AA)$ and the maximum absolute entry $\Delta_1(\AA)$.
\begin{cor}
\label{cor_gap_immediate} Assume the conditions of Theorem~\ref{thm_transference_gap_integer_hull}.  Then
the bounds
\begin{equation}
\IG(\AA,\bb,\cc)\le \frac{\; s\binom{s+m}{m}^{1/2}}{2^{s-m-1}}\cdot\frac{\Delta_m(\AA)}{\gcd(\AA)}\label{eq:Gap_transference_Delta_m}
\end{equation}
and 
\begin{equation}
\IG(\AA,\bb,\cc)\le \frac{\; s(s+m)^{m/2}}{2^{s-m-1}}\cdot\frac{\,(\Delta_1(\AA))^m}{\gcd(\AA)}\label{eq:Gap_transference_Delta_1}
\end{equation}
hold.
\end{cor}
For a unit cost vector ${\ve c}$, the bound \eqref{eq:Gap_transference_Delta_m} improves on \eqref{IG_C} when $s\geq 4m$, and the bound \eqref{eq:Gap_transference_Delta_1} improves on \eqref{IG_EW} when $s \geq 6m$.

The proof of Theorem~\ref{thm_transference_gap_integer_hull} makes use of results from convex geometry and  Minkowski's geometry of numbers. Following a similar approach, we obtain a new upper bound for the $\ell_2$-distance from 
a vertex of the polyhedron $\KP(\AA, {\ve b})$ to its nearest integer point in $\KP(\AA, {\ve b})$. 

\begin{thm}
\label{thm:UpperBound}
Let $\AA \in \Z^{m\times n}$, $m<n$, be a matrix of rank $m$, $\ve b \in \Z^m $, and suppose that $ \KP(\AA, {\ve b})$ contains integer points.  Let ${\ve x}^*$ be a vertex of  $ \KP(\AA, {\ve b})$. There
exists an integer point ${\ve z}^*\in \KP(\AA, {\ve b})$ such that
\begin{equation}\label{t:Bound_via_pencils}
\|{\ve x}^*-{\ve z}^*\|_2\le  \frac{\Delta(\AA)}{\gcd(\AA)}-1\,.
\end{equation}
\end{thm}
We remark that for certain matrices $\AA$ the bound \eqref{t:Bound_via_pencils} is smaller than the $\ell_2$-norm proximity bounds $\sqrt{n}(n-m)\Delta_m(\AA)$  and $m (2m\Delta_1(\AA)+1)^m$ that can be derived from \cite{CKPW} and \cite{MR3775840}, respectively.  It is sufficient to observe that the ratio $\Delta(\AA)/\Delta_m(\AA)$ can be arbitrarily close to one and the ratio  $\Delta(\AA)/\Delta_1(\AA)$ can be arbitrarily small.

Applying Theorem \ref{thm:UpperBound} to a vertex optimal solution ${\ve x}^*$ of \eqref{initial_LP} we obtain the following estimate.

\begin{cor}
\label{cor_gap_via_length} Let $\AA\in \Z^{m\times n}$, $m<n$, be a matrix of rank $m$, $\bb\in \Z^m$ and $\cc\in \R^n$ be a unit cost vector. Suppose that (\ref{initial_IP}) is feasible and bounded. Then the bound 
\be\label{Gap_zonotope} 
\IG(\AA,\bb,\cc)\le\frac{\Delta(\AA)}{\gcd(\AA)}-1
\ee
holds.
\end{cor}
For a unit cost vector ${\ve c}$, the bound \eqref{Gap_zonotope} improves   \eqref{IG_C} when $\Delta(\AA)/\gcd(\AA)<(n-m)\Delta_m(\AA)+1$ and improves  \eqref{IG_EW}  when $\Delta(\AA)/\gcd(\AA)<mn^{-1/2} (2m\Delta_1(\AA)+1)^m+1$.

\section{Volumes and linear transforms}

In this section, we develop geometric tools needed for the proof of the transference bounds. 

For a matrix $\AA\in \R^{l\times r}$ we denote by $\lin(\AA)$ the linear subspace of $\R^l$ spanned by the columns of $A$. Given a set $M\subset \R^r$ we let $\AA M=\{\AA{\ve x}\in \R^l: {\ve x}\in M\}$ and use the notation $[\AA]=\AA[0,1]^r$. For a set $X\subset \R^l$ and a linear subspace $L$ of $\R^l$, we denote by $X|L$ the orthogonal projection of $X$ onto $L$. Further,  $\vol_i(\cdot)$  denotes the $i$-dimensional volume.

Let $S$ be an $(l-k)$-dimensional subspace of $\R^l$. Consider an orthonormal basis ${\ve s}_1, \ldots, {\ve s}_{l-k}, {\ve s}_{l-k+1}, \ldots, {\ve s}_l$ of $\R^l$ such that the first $l-k$ vectors form a basis of $S$. Let further $\SS_{l-k}=({\ve s}_1, \ldots, {\ve s}_{l-k})\in \R^{l\times (l-k)}$ and $\SS_{k}=({\ve s}_{l-k+1}, \ldots, {\ve s}_{l})\in \R^{l\times k}$. 

Given a measurable set  $M$ in the subspace $S$, we are interested in the $(l-k)$-dimensional volume of its image 
$\DD M$, where we assume that $\DD\in \R^{l\times l}$ is an invertible matrix. The first result gives a general expression for $\vol_{l-k}(\DD M)$ in terms of $\vol_{l-k}(M)$.

\begin{lemma} \label{section_volume_regular} Let $S$ be an $(l-k)$-dimensional subspace of $\R^l$. Let $M\subset S$ be measurable, $\DD\in \R^{l\times l}$ be nonsingular
and let the rows of $\BB\in \R^{k\times l}$ form a basis of the subspace $(\DD S)^\bot$, the orthogonal complement of the subspace $\DD S=\lin(\DD \SS_{l-k})$. Then
\bea
\vol_{l-k}(\DD M)= |\det(\DD)| \sqrt{\frac{\det(\BB \BB^\top)}{\det(\BB\DD\DD^\top \BB^\top)}}\vol_{l-k}(M)\,.
\eea
\end{lemma}

\begin{proof}

By the elementary properties of volume, we have
\be\label{volume_as_product}
\vol_{l-k}(\DD M)=\vol_{l-k}([\DD \SS_{l-k}])\vol_{l-k}(M)\,.
\ee
On the other hand, we have 
\be\label{det_expression}
\begin{aligned}
|\det(\DD)| & =\vol_l([\DD (\SS_{l-k}, \SS_k)])\\
& = \vol_{l-k}([\DD \SS_{l-k}])\vol_k([\DD\SS_k]| \lin(\DD\SS_{l-k})^\bot)\\
& = \vol_{l-k}([\DD \SS_{l-k}])\vol_k([\DD\SS_k]| \lin(\BB^\top))\,.
\end{aligned}
\ee
Now, according to the definition of projections, we have
\bea
[\DD\SS_k]|\lin(\BB^\top)=\BB^\top (\BB \BB^\top)^{-1}\BB[\DD\SS_k]
\eea
and hence
\bea
\begin{aligned}
\vol_k([\DD\SS_k]|\lin(\BB^\top))&=\sqrt{ \det(\SS_k^\top \DD^\top\BB^\top(\BB\BB^\top)^{-1}\BB\BB^\top(\BB\BB^\top)^{-1}\BB\DD\SS_k) }\\
&=\sqrt{\det((\BB\BB^\top)^{-1})}\sqrt{\det(\BB\DD\SS_k\SS_k^\top\DD^\top\BB^\top)}\\
&=\sqrt{\frac{\det(\BB\DD\DD^\top\BB^\top)}{\det(\BB\BB^\top)}}\,.
\end{aligned}
\eea
Substituting this expression in \eqref{det_expression} gives along \eqref{volume_as_product}
the identity.
\qed
\end{proof}

The second result provides a lower bound for $\vol_{l-k}(\DD M)$ that involves the eigenvalues of the matrix $\DD^\top \DD$ and $\vol_{l-k}(M)$.

\begin{lemma} \label{volume_via_eigenvalues}
Let $M\subset S$ be measurable, $\DD\in\R^{l\times l}$ nonsingular and let $\lambda_1\le \lambda_2\le \cdots \le \lambda_l$
be the positive eigenvalues of $\DD^\top \DD$. Then
\bea \vol_{l-k}(\DD M)\ge \left(\prod_{i=1}^{l-k} \sqrt{\lambda_i}\right) \vol_{l-k}(M)\,.\eea
\end{lemma}

\begin{proof}
According to \eqref{volume_as_product}, we need to estimate 
\bea
\vol_{l-k}([\DD \SS_{l-k}])=\sqrt{\det(\SS^\top_{l-k}\DD^\top\DD \SS_{l-k})}\,.
\eea
Let $\OO\in\R^{l\times l}$ be a matrix such that its rows form an orthonormal basis consisting of eigenvectors 
of $\DD^\top \DD$.  For convenience, we will denote by $\diag(\lambda_i)$ the diagonal matrix with the eigenvalues $\lambda_1,\ldots, \lambda_l$ on the main diagonal. Then we have $\DD^\top \DD=\OO^\top \diag(\lambda_i)\OO$. 

Let $[l]= \{1,\ldots,l\}$ and let $\binom{[l]}{r}$ be the set of all $r$-element subsets of $[l]$. 
With $\widetilde{\SS_{l-k}}=\OO\SS_{l-k}$ we get by the Cauchy-Binet formula
\bea
\begin{aligned}
\det(\SS^\top_{l-k}\DD^\top \DD \SS_{l-k}) &= \det(\widetilde{\SS_{l-k}}^\top \diag(\lambda_i^{1/2})\diag(\lambda_i^{1/2})\widetilde{\SS_{l-k}})\\
&=\sum_{I\in {[l]\choose l-k}}\det(\diag(\lambda_i^{1/2})^I)^2 (\det(\widetilde{\SS_{l-k}}^I))^2\\
&\ge \left(\prod_{i=1}^{l-k}\lambda_i\right) \sum_{I\in {[l]\choose
    l-k}}(\det(\widetilde{\SS_{l-k}}^I))^2\\
&= \left(\prod_{i=1}^{l-k}\lambda_i\right) \det(\widetilde{\SS_{l-k}}^\top \widetilde{\SS_{l-k}} )= \prod_{i=1}^{l-k}\lambda_i\,.
\end{aligned}
\eea
Here $\diag(\lambda_i^{1/2})^I$ is the $(l-k)\times(l-k)$ diagonal matrix with $\lambda_i^{1/2}$ indexed by $I$
and $\widetilde{\SS_{l-k}}^I$ is the $(l-k)\times(l-k)$ submatrix of
$\widetilde{\SS_{l-k}}$ with rows indexed by $I$. In the second to
last identity we used again the Cauchy-Binet formula.

\qed\end{proof}

From Lemma \ref{volume_via_eigenvalues} we  obtain the following corollary.

%

\begin{cor} \label{coro_section_volume_via_B} Let $S$ be an $(l-k)$-dimensional subspace of $\R^l$ and let  $\DD=\diag(d_1,\ldots, d_l)$ with $0<d_1\le d_2\le \cdots\le d_l$. Then
\be\label{Section_volume_via_the product} \vol_{l-k}(\DD(S\cap (-1,1)^l))\ge 2^{l-k}\prod_{i=1}^{l-k} d_i\,.\ee
\end{cor}

\begin{proof}
By Vaaler's cube slicing inequality \cite{Vaaler1979}, we have $\vol_{l-k}(S\cap (-1,1)^l)\ge 2^{l-k}$.
Now the bound \eqref{Section_volume_via_the product} immediately follows from Lemma \ref{volume_via_eigenvalues}.

\qed\end{proof}

%

\section{Proofs of the transference bounds}

We will derive Theorem  \ref{thm_transference_gap_integer_hull} from Corollary \ref{cor_gap_via_length}  and from the following result. 
\begin{thm} \label{thm_gap_integer_hull} Let $\AA\in \Z^{m\times n}$, with $n>m+1$, be a matrix of rank $m$, $\bb\in \Z^m$ and $\cc\in \R^n$ be a unit cost vector. Suppose that (\ref{initial_IP}) is feasible and bounded. Let $\zz^*=(z_1^*,\ldots, z_n^*)^\top$ be an optimal solution of (\ref{initial_IP}) which is a vertex of $\KP_I(\AA,\bb)$. Assuming without loss of generality $z_1^*\le\cdots\le z_n^*$, the bound
\be\label{Gap_transference_integer_hull_advanced} 
\IG(\AA,\bb,\cc)\le\frac{(n-m)}{\prod_{i=1}^{n-m-1}(z_i^*+1)}\cdot\frac{\Delta(\AA)}{\gcd(\AA)}
\ee
holds.
\end{thm}

\begin{proof}


Let  ${\ve c}|{\ker(\AA)}$  denote the orthogonal projection of the vector $\cc$ on the kernel subspace $\ker(\AA)=\{{\ve x}\in\R^n: \AA{\ve x}={\ve 0}\}$ of the matrix $\AA$. Observe first that if ${\ve c}$ is orthogonal to $\ker(\AA)$, then $\IG(\AA,\bb,\cc)=0$ and the bound \eqref{Gap_transference_integer_hull_advanced} holds. Hence, we may assume without loss of generality that ${\ve c}|{\ker(\AA)}$ is a nonzero vector.

Suppose, to derive a contradiction, that the bound \eqref{Gap_transference_integer_hull_advanced} does not hold. Then there exists a vertex ${\ve x}^*$ of $\KP(\AA,\bb)$ and
a vertex ${\ve z}^*$ of $P_I(\AA, {\ve b})$ optimising \eqref{initial_IP}, such that, assuming $z_1^*\le\cdots\le z_n^*$, we have
\be\label{opposite_bound_integer_hull} 
{\ve c}\cdot({\ve z}^*-{\ve x}^*)> \frac{(n-m)}{\prod_{i=1}^{n-m-1}(z_i^*+1)}\cdot\frac{\Delta(\AA)}{\gcd(\AA)}\,.
\ee

Let $d_i=z_i^*+1$, $i\in[n]$, $\DD=\diag(d_1,\ldots, d_n)$ and let $\BB\in\Z^{(m+1)\times n}$ be the matrix obtained by adding the $(m+1)$-st row ${\ve c}^\top$ to the matrix $\AA$. Let further $V=\ker(\BB)$. Consider the box section
\bea
K=V\cap (-d_1, d_1)\times\cdots\times (-d_n, d_n)\,.
\eea
We can write $K=\DD M$, where $M$ is a $(n-m-1)$-dimensional section of the cube $(-1,1)^n$.
Hence, by Corollary \ref{coro_section_volume_via_B},
\be\label{vol_K_lower_bound}
\vol_{n-m-1}(K)\ge 2^{n-m-1}\prod_{i=1}^{n-m-1} d_i\,.
\ee
Consider
the origin-symmetric convex set
\bea
L=\conv({\ve x}^*-{\ve z}^*, K, {\ve z}^*-{\ve x}^*)\subset \ker(\AA)\,.
\eea
The set $L$ is a bi-pyramid with apexes $\pm ({\ve x}^*-{\ve z}^*)$ and 
$(n-m-1)$-dimensional basis $K$. As $K\subset \ker(\AA)\cap\ker(\ve c)$
the height of ${\ve x}^*-{\ve z}^*$ over $K$ is given by
\bea
 \frac{{\ve c}|\ker(\AA)\cdot({\ve z}^*-{\ve x}^*)}{\|{\ve c}|{\ker(\AA)}\|_2} =\frac{{\ve c}\cdot({\ve z}^*-{\ve x}^*)}{\|{\ve c}|{\ker(\AA)}\|_2}.
\eea 
Hence, we have 
\bea
\vol_{n-m}(L)=\frac{2 \,{\ve c}\cdot({\ve z}^*-{\ve x}^*)\,\vol_{n-m-1}(K)}{(n-m)\|{\ve c}|{\ker(\AA)}\|_2}\,.\eea
Then, using \eqref{vol_K_lower_bound} and \eqref{opposite_bound_integer_hull} and noting that the assumption $\|{\ve c}\|_2= 1$ implies $\|{\ve c}|{\ker(\AA)}\|_2\le 1$, we obtain the lower bound
\be\label{volume_bound_for_L}
\vol_{n-m}(L)>2^{n-m}\frac{\Delta(\AA)}{\gcd(\AA)}\,.
\ee
Observe that the lattice $\Lambda(\AA)=\ker(\AA)\cap\Z^n$ has determinant 
\be\label{lat_det_lambda}
\det(\Lambda(\AA))=\frac{\Delta(\AA)}{\gcd(\AA)}\,
\ee
(see e. g. \cite[Chapter~1, \S 1]{Skolem}). The $(n-m)$-dimensional subspace $\ker(\AA)$ can be considered as a usual Euclidean $(n-m)$-dimensional space.  Therefore, by \eqref{volume_bound_for_L}, \eqref{lat_det_lambda} and Minkowski's first fundamental theorem (in the form of Theorem II in Chapter III of \cite{cassels1996introduction}), applied to the set $L$ and the lattice $\Lambda(\AA)$,  there is a nonzero point ${\ve y}\in L \cap \Lambda(\AA)$. 

Suppose first that ${\ve c}\cdot{\ve y}=0$. Consider the points ${\ve y}^+={\ve z}^*+ {\ve y}$ and ${\ve y}^-={\ve z}^*- {\ve y}$. We have  ${\ve y}^+, {\ve y}^-\in ({\ve z}^*+K)$ and, consequently,  ${\ve y}^+, {\ve y}^-\in \KP(\AA, {\ve b})$.  Further, ${\ve z}^*$ is the midpoint of the segment with endpoints ${\ve y}^+$ and ${\ve y}^-$, contradicting the choice of ${\ve z}^*$ as a vertex of the integer hull $P_I(\AA, {\ve b})$. 

It remains to consider the case ${\ve c}\cdot{\ve y}\neq 0$. Since $L$ is origin-symmetric, we may assume without loss of generality that ${\ve c}\cdot{\ve y}<0$. Observe that the point ${\ve y}^+={\ve z}^*+ {\ve y}$ is in the set $\conv({\ve z}^*+K,{\ve x}^*)$ and hence ${\ve y}^+\in  \KP(\AA, {\ve b})$. Now, it is sufficient to notice that  ${\ve c}\cdot{\ve y}^+<{\ve c}\cdot {\ve z}^*$, contradicting the optimality of ${\ve z}^*$.

\qed\end{proof}

\subsection{Proof of Theorem \ref{thm_transference_gap_integer_hull} }

 Let $I=\{i_1,\ldots, i_k\} \subset  [n]$ with $i_1<i_2<\cdots<i_k$. We will use the notation  $\AA_I$ for the $m \times k$ submatrix  of $\AA$ with columns indexed by $I$.  In the same manner, given ${\ve x}\in \R^n$, we will denote by ${\ve x}_I$ the vector $(x_{i_1}, \ldots, x_{i_k})^{\top}$. By $\R^I$ we denote the $k$-dimensional real space with coordinates indexed by $I$. 

Let $\xx^{*}$ be a vertex optimal solution to (\ref{initial_LP}). Clearly, we may assume that $\xx^{*}\neq \zz^{*}$.
Let $I\subset[n]$ denote the set of indices $i$ for which at least one of $z_i^*,x_i^*$ is non-zero. We consider a new linear program
\be\label{new_linear_prog}
\min\{\cc_I\cdot\xx_I:\AAhat_I\xx_I=\hat{\bb},\xx_I\geq 0\}\,,
\ee
where ${\AAhat}_I\in\Z^{\hat{m}\times\hat{n}}$ is a full-row-rank matrix with $\gcd({\AAhat}_I)=1$ and $\hat{\bb}\in\Z^{\hat{m}}$ such that ${\AAhat}_I\xx_I=\hat{\bb}$ and $\AA_I\xx_I=\bb$ decribe the same affine subspace in $\R^I$. Note that $\xx_I^*$ and $\zz_I^*$ are optimal fractional and integral solutions, respectively, and that $\zz_I^*$ is a vertex of the integer hull of $\KP({\AAhat}_I,\hat{\bb})$. Note also that
\be\label{new_delta_lower_bound}
\Delta({\AAhat}_I)\leq\frac{\Delta(\AA)}{\gcd(\AA)}\,,
\ee
one can see this by observing that the quantity on the left is a divisor of the volume of the orthogonal projection of a parallelepiped whose volume is given by the quantity on the right. 

If $\hat{n}=\hat{m}+1$, then the bound \eqref{Gap_transference_integer_hull} immediately follows from the bound 
\eqref{Gap_zonotope} in Corollary \ref{cor_gap_via_length}. Otherwise, suppose that $\hat{n}>\hat{m}+1$. We have then that ${\AAhat}_I$, $\hat{\bb}$, $\cc_I/\norm{\cc_I}_2$, $\xx_I^*$, $\zz_I^*$ satisfy the hypotheses of Theorem~\ref{thm_gap_integer_hull}. We therefore get
\be\label{new_IG_bound}
\IG(\AA,\bb,\cc)=\IG(\AAhat_I,\hat{\bb},\cc_I)\leq\norm{\cc_I}_2\cdot\frac{\hat{n}-\hat{m}}{\prod_i (z_i^*+1)}\cdot\Delta(\AAhat_I),
\ee
where the product in the denominator is over the $\hat{n}-\hat{m}-1$ smallest coordinates of $\zz_I^*$.

Now, $\hat{n}$ is equal to $\norm{\zz^*+\xx^*}_0$. Also, $\xx^*$ has support of size at most $\hat{m}$, since $\norm{\xx^*}_0=\norm{\xx_I^*}_0$ and $\xx_I^*$ is a vertex of the new linear program \eqref{new_linear_prog}. Thus we get
\be
\label{numerator_upper_bound}
\hat{n}-\hat{m}\leq\norm{\zz^*+\xx^*}_0-\norm{\xx^*}_0\leq\norm{\zz^*}_0=s.
\ee
On the other hand, we have the following lower bound for the product in the denominator of \eqref{new_IG_bound}:
\be
\label{denominator_lower_bound}
\prod_i (z_i^*+1)\geq 2^{\hat{n}-\hat{m}-1-\left|I\,\backslash\,\text{Supp}(\zz_I^*)\right|}
=2^{s-\hat{m}-1}
\geq 2^{s-m-1}\,.
\ee
Combining together \eqref{new_IG_bound}, \eqref{numerator_upper_bound}, \eqref{denominator_lower_bound}, and the fact that $\norm{\cc_I}_2\leq1$, we get
\be\label{new_IG_bound_combined}
\IG(\AA,\bb,\cc)\leq\frac{s}{2^{s-m-1}}\cdot\Delta(\AAhat_I)\,.
\ee
The desired conclusion \eqref{Gap_transference_integer_hull} then follows from \eqref{new_delta_lower_bound} and \eqref{new_IG_bound_combined}. 

\subsection{Proof of Corollary \ref{cor_gap_immediate}}

Let $I$ and $\AAhat_I$ be as in the proof of Theorem~\ref{thm_transference_gap_integer_hull}. We will derive the bounds \eqref{eq:Gap_transference_Delta_m} and \eqref{eq:Gap_transference_Delta_1}
from the bound \eqref{new_IG_bound_combined}. Choose any $J\subset [n]\backslash I$ that is minimal with respect to the property that $\AA_{I\cup J}$ has rank $m$. Thus, $\left|J\right|=m-\hat{m}$, and \be
\label{size_of_I_cup_J}
\left|I\cup J\right| = \hat{n}+(m-\hat{m}) \leq s+\hat{m} +(m-\hat{m}) = s+m.
\ee
As with \eqref{new_delta_lower_bound}, we have
\be\label{new_delta_lower_bound_I_cup_J}
\Delta({\AAhat}_I)\leq\frac{\Delta(\AA_{I\cup J})}{\gcd(\AA_{I\cup J})}\leq\frac{\Delta(\AA_{I\cup J})}{\gcd(\AA)}\,.
\ee
By \eqref{size_of_I_cup_J} and the Cauchy-Binet formula, we have
\be\label{Cauchy-Binet-cor}
\Delta(\AA_{I\cup J})\le {\left|I\cup J\right|\choose m}^{1/2}\Delta_m(\AA_{I\cup J})\le {s+m\choose m}^{1/2}\Delta_m(\AA)\,. 
\ee
Combining \eqref{new_IG_bound_combined}, \eqref{new_delta_lower_bound_I_cup_J}, and \eqref{Cauchy-Binet-cor}, we obtain the bound \eqref{eq:Gap_transference_Delta_m}. On the other hand, if we use \eqref{size_of_I_cup_J} and Hadamard's inequality, we get
\be\label{Hadamard-cor}
\Delta(\AA_{I\cup J})\le (\sqrt{\left|I\cup J\right|}\cdot\Delta_1(\AA_{I\cup J}))^m\le (\sqrt{s+m}\cdot\Delta_1(\AA))^m\,.
\ee
Combining \eqref{new_IG_bound_combined}, \eqref{new_delta_lower_bound_I_cup_J}, and \eqref{Hadamard-cor}, we obtain the bound \eqref{eq:Gap_transference_Delta_1}.

\section{Proof of the $\ell_2$-distance proximity bound}

First, we will prove two lemmas needed for the proof of Theorem \ref{thm:UpperBound}.

Let ${\ve y} \in\R^n$ and let
\bea
C^n({\ve y})=\{{\ve x}\in  \R^n: \|{\ve x}-{\ve y} \|_{\infty} < 1 \}
\eea
be an open cube in $\R^n$ with edge length $2$ centered at the point ${\ve y}$. 
Given two points ${\ve u}, {\ve v}\in \R^n$ we will
consider an open set $D({\ve u}, {\ve v})$ defined as
\bea
D({\ve u}, {\ve v})=\conv(C^n({\ve u}), C^n({\ve v}))\,.
\eea

\begin{lemma}
\label{D_property} Let ${\ve u}, {\ve v}\in\R^n_{\ge 0}$. Then
$D({\ve u}, {\ve v})\cap\Z^n=D({\ve u}, {\ve v})\cap\Z^n_{\ge 0}\,.$
\end{lemma}
\begin{proof}
Suppose, to derive a contradiction, that there exists an integer point ${\ve z}=(z_1,\ldots, z_n)^\top \in D({\ve u}, {\ve v})$ such that $z_j\le -1$ for some $j\in [n]$.
Then there exist points ${\ve x}=(x_1,\ldots, x_n)^\top\in C^n({\ve u})$ and  ${\ve y}=(y_1,\ldots, y_n)^\top\in  C^n({\ve v})$ such that for some $\lambda\in [0,1]$ 
\bea
z_j=\lambda x_j + (1-\lambda) y_j\,.
\eea
Therefore, since $x_j > -1$ and $y_j> -1$, we must have $z_j>-1$. 
\qed\end{proof}

Next, we consider an origin-symmetric open convex set $E=E({\ve u}, {\ve v})$ defined as
\bea
E=\conv(C^n({\ve u}-{\ve v}), C^n({\ve v}-{\ve u}))\,.
\eea
Notice that
\be\label{zonotope_union}
E= (D({\ve u}, {\ve v})-{\ve v})\cup(-D({\ve u}, {\ve v})+{\ve v})\,.
\ee

\begin{lemma}
\label{E_property} Suppose that ${\ve u},{\ve v}\in \KP(\AA, {\ve b})$. Then the bound 
\be\label{E_property_bound} 
\vol_{n-m}(E\cap\ker(\AA))\ge 2^{n-m}(1+\|{\ve u}-{\ve v}\|_2)
\ee
holds.
\end{lemma}
\begin{proof}
First, we will separately consider the case $n=m+1$. Then $\ker(\AA)$ has dimension one
and, noticing that ${\ve u}-{\ve v}\in \ker(\AA)$ we can write
\bea
\vol_{1}(E\cap\ker(\AA))=\vol_{1}(C^n({\ve 0})\cap\ker(\AA))+2\|{\ve u}-{\ve v} \|_2\,.
\eea
Since $\vol_{1}(C^n({\ve 0})\cap\ker(\AA))\ge 2$, we obtain the bound \eqref{E_property_bound}.

For the rest of the proof we assume that $n>m+1$. 
If ${\ve u}={\ve v}$, then $E=C^n({\ve 0})$ and the bound \eqref{E_property_bound} immediately follows from Vaaler's cube slicing inequality \cite{Vaaler1979}. Hence, we may also assume without loss of generality that ${\ve u}-{\ve v}\neq {\ve 0}$.

Let
\bea
S=C^n({\ve 0})\cap \ker(\AA)\cap \ker(({\ve u}-{\ve v})^\top)\,.
\eea
The set $S$ is a section of the open cube $C^n({\ve 0})$. Since ${\ve u}-{\ve v}\in \ker(\AA)\setminus \{\ve 0\}$, the section $S$ has dimension  $n-m-1$.
Let further 
\bea
S^+=\{{\ve x}\in C^n({\ve 0})\cap \ker(\AA): ({\ve u}-{\ve v})\cdot{\ve x}>0\}
\eea
and 
\bea
S^-=\{{\ve x}\in C^n({\ve 0})\cap \ker(\AA): ({\ve u}-{\ve v})\cdot{\ve x}<0\}\,.
\eea
By construction, $S^+$, $S^-$ and $S$ do not overlap and $C^n({\ve 0})\cap \ker(\AA)=S^+\cup S^-\cup S$. Further, by Vaaler's cube slicing inequality \cite{Vaaler1979}, we have
\be\label{volume_of_union}
\vol_{n-m}(S^+)+\vol_{n-m}(S^-)=\vol_{n-m}(C^n({\ve 0})\cap \ker(\AA))\ge 2^{n-m}
\ee
and
\be\label{volume_of_section_S}
\vol_{n-m-1}(S)\ge 2^{n-m-1}\,.
\ee
Observe that $E\cap\ker(\AA)$ contains the sets $S^+ + {\ve u}-{\ve v}$,  $S^- + {\ve v}-{\ve u}$ and the cylinder
$\conv( {\ve u}-{\ve v} +S,   {\ve v}-{\ve u}+S)$. These three sets do not overlap and, using \eqref{volume_of_union} and \eqref{volume_of_section_S},
we have
\bea\begin{aligned}
\vol_{n-m}(E\cap\ker(\AA))& \ge \vol_{n-m}(S^+)+\vol_{n-m}(S^-)\\
&+ \vol_{n-m}(\conv( {\ve u}-{\ve v} +S,   {\ve v}-{\ve u}+S))\\
& =\vol_{n-m}(  C^n({\ve 0})\cap \ker(\AA)) + 2\,\vol_{n-m-1}(S)\,\|{\ve u}-{\ve v} \|_2\\
& \ge 2^{n-m}(1+\|{\ve u}-{\ve v} \|_2)\,.
\end{aligned}
\eea
Hence, we obtain the bound \eqref{E_property_bound}.
\qed

\end{proof}

\subsection{Proof of Theorem \ref{thm:UpperBound}}

We will say that $B\subset [n]$ is a {\em basis} of $\AA$ if $|B|=m$ and the submatrix $\AA_B$  is nonsingular. %
Take any vertex ${\ve x}^*\in \KP(\AA, {\ve b})$. There is a basis $B$ of $\AA$ such that, denoting by $N$ the complement of $B$ in $[n]$, we have
\bea 
 {\ve x}^*_B = \AA_B^{-1}{\ve b} \mbox{ and }{\ve x}^*_{N}={\ve 0}_{N}\,.
\eea

Choose an integer point ${\ve z}^*\in \KP(\AA, {\ve b})$ with the minimum possible distance between the points ${\ve x}^*_N={\ve 0}_{N}$ and ${\ve z}^*_N$. Then
\be\label{minimum_distance_choice}
\|{\ve z}^*_N\|_2=\min \{\|{\ve y}_N\|_2: {\ve y}\in \KP(\AA, {\ve b})\cap\Z^n\}\,.
\ee 
Suppose, to derive a contradiction, that the bound \eqref{t:Bound_via_pencils} does not hold for the point ${\ve z}^*$. Then, using \eqref{lat_det_lambda},
\be\label{lower_bound_distance}
\|{\ve x}^*-{\ve z}^*\|_2> \frac{\Delta(\AA)}{\gcd(\AA)}-1 = \det(\Lambda(\AA))-1\,.
\ee
Recall that we denote by $\Lambda(\AA)$ the  lattice formed by all integer points in the kernel subspace of the matrix $\AA$.

The lower bound \eqref{lower_bound_distance} and Lemma \ref{E_property}  imply that for $E=E({\ve x}^*, {\ve z}^*)$
we have
\be\label{vol_E_kerA_bound}
\vol(E\cap\ker(\AA))> 2^{n-m} \det(\Lambda(\AA))\,.
\ee
The $(n-m)$-dimensional subspace $\ker(\AA)$ can be considered as a usual Euclidean $(n-m)$-dimensional space. 
 Noting the bound \eqref{vol_E_kerA_bound}, Minkowski's first fundamental theorem (in the form of Theorem II in Chapter III of \cite{cassels1996introduction}) implies that the set $E\cap\ker(\AA)$
 contains nonzero points $\pm {\ve z}$ of the lattice $\Lambda(\AA)$. Using \eqref{zonotope_union}, we may assume without loss of generality
 that we have  ${\ve z}\in D({\ve x}^*, {\ve z}^*)-{\ve z}^*$. Therefore, the point ${\ve w}={\ve z}+{\ve z}^*$ is in the set $D({\ve x}^*, {\ve z}^*)\cap (\ker(\AA)+{\ve z}^*)$. By Lemma \ref{D_property}, ${\ve w}\in \Z^n_{\ge 0}$ and, hence, ${\ve w}\in \KP(\AA, {\ve b})$. 
 
 Next, we will show that
 $\|{\ve w}_N\|_2< \|{\ve z}^*_N\|_2$, contradicting \eqref{minimum_distance_choice}.
Notice first that for any ${\ve x}\in \KP(\AA, {\ve b})$ we have ${\ve x}_B=\AA_B^{-1}({\ve b}-\AA_N{\ve x}_N)$. Hence ${\ve w}_N = {\ve z}^*_N$ implies that  ${\ve w}={\ve z}^*$.
Therefore, we may assume that ${\ve w}_N \neq {\ve z}^*_N$. 
%

Take any index $j\in N$. Since ${\ve w}\in D({\ve x}^*, {\ve z}^*)$, 
we have $w_j\le z^*_j$. Hence there is at least one index $j_0\in N$ with $w_{j_0}< z^*_{j_0}$. Therefore  $\|{\ve w}_N\|_2< \|{\ve z}^*_N\|_2$ and we obtain a contradiction with \eqref{minimum_distance_choice}.



\section{Acknowledgement}

The authors thank the anonymous referees for their valuable comments and suggestions.

\end{document}